\documentclass[12pt,reqno]{amsart}

\usepackage{graphicx}
\usepackage[arrow,matrix,curve]{xy} 

\usepackage{amssymb, latexsym, amsmath, amscd, array, 
}

\newtheorem{theorem}{Theorem}[section]
\newtheorem{lemma}[theorem]{Lemma}
\newtheorem{proposition}[theorem]{Proposition}
\newtheorem{corollary}[theorem]{Corollary}

\theoremstyle{definition}
\newtheorem{definition}[theorem]{Definition}

\newtheorem{remark}[theorem]{Remark}
\newtheorem{question}[theorem]{Question}

\numberwithin{equation}{section}
\numberwithin{figure}{section} 
\numberwithin{table}{section}

\DeclareMathOperator{\PD}{{\rm PD}}

\DeclareMathOperator{\DD}{\mathcal{D}}

\DeclareMathOperator{\sys}{{\rm sys}}
\DeclareMathOperator{\stsys}{{\rm stsys}}

\DeclareMathOperator{\Isom}{{\rm Isom}}

\DeclareMathOperator{\capa}{{\rm Cap}} 

\DeclareMathOperator{\area}{{\rm area}}
\DeclareMathOperator{\vol}{{\rm vol}}

\newcommand\ie {{\it i.e.\ }}

\newcommand\cf {\hbox{\it cf.}}

\newcommand\dist{{\rm dist}}

\newcommand{\length}{{\rm length}}

\newcommand \cqfd{\unskip\kern 6pt\penalty 500
\raise -2pt\hbox{\vrule\vbox to10pt{\hrule width 4pt
\vfill\hrule}\vrule}\par}                 

\def\adots{\mathinner{\mkern2mu\raise1pt\hbox{.}
\mkern3mu\raise4pt\hbox{.}\mkern1mu\raise7pt\hbox{.}}}
\def\hfl#1{\frac{\buildrel{#1}}{{\hbox to 12mm{\rightarrowfill}}}}
  
\newcommand\C{{\mathbb {C}}}

\newcommand\Q{{\mathbb {Q}}} \newcommand\R {{\mathbb R}}
\newcommand\RR {{\mathbb R}}
\newcommand\RP {{\mathbb R}{\mathbb P}}
\newcommand\T {{\mathbb T}} \newcommand\Z {{\mathbb Z}}

\newcommand{\rp}{\mathbb R\mathbb P} 
\newcommand\CP {{\mathbb C\mathbb P}}

\long\def\forget#1\forgotten{} %

\long\def\forgett#1\forgottent{} %

\def\circ{\mathchoice%
 {\mathrel{\raise 1pt\hbox{$\scriptstyle\mathchar"020E$}}}
 {\mathrel{\raise 1pt\hbox{$\scriptstyle\mathchar"020E$}}}
 {\mathrel{\raise 1pt\hbox{$\scriptscriptstyle\mathchar"020E$}}}
 {}
}

\newcommand{\nc}{\newcommand} \nc{\on}{\operatorname}

\nc{\df}{\on{\it df}}

\nc{\conf}{\on{conf}}

\nc{\spt}{\on{spt}}
\nc{\norm}[1]{\| #1 \|}
\nc{\parallelleer}{\norm{\ }} 
\nc{\parallelh}{\norm h} 
\nc{\parallelk}{\norm k} 
\nc{\parallelx}{\norm x} 
\nc{\parallelhrr}{\norm {h_\RR}} 
\nc{\parallelom}{\norm \omega} 
\nc{\parallelomij}{\norm {\omega_{i_j}}} 
\nc{\parallelomx}{\norm {\omega_{x}}} 
\nc{\parallelpi}{\norm \pi} 
\nc{\parallelalf}{\norm \alpha} 
\nc{\parallelalfs}{\norm {\alpha_s}} 
\nc{\parallelalfi}{\norm {\alpha_i}} 
\nc{\parallelalfij}{\norm {\alpha_{i_j}}} 
\nc{\parallelbeta}{\norm \beta} 
\nc{\parallelbetat}{\norm {\beta_t}} 
\nc{\parallelhcapalf}{\norm {h \cap \alpha}} 
\nc{\parallelPDralf}{\norm {\PD_\RR(\alpha)}} 
\nc{\strichleer}{| \  |}
\nc{\NN}{\mathbb N}

\nc{\rr}{\mbox{$\scriptstyle\mathbb R$}}

\nc{\dF}{{\it dF}} 

\nc{\DF}{{\it DF}} 

\nc{\ds}{{\it ds}} 

\nc{\dvol}{{\it dvol}}

\nc{\grad}{{\rm grad}} 
\nc{\strichw}{\|\omega\|} 
\nc{\strichwx}{|\omega_x|}
\nc{\Hess}{{\rm Hess}}

\begin{document}

\title{Dyck's surfaces, systoles, and capacities}


\author[M.~Katz]{Mikhail G. Katz}

\address{Department of Mathematics, Bar Ilan University, Ramat Gan
52900 Israel} 

\email{katzmik@macs.biu.ac.il}

\author[S.~Sabourau]{St\'ephane Sabourau}


\address{Laboratoire d'Analyse et Math\'ematiques Appliqu\'ees, 
Universit\'e Paris-Est Cr\'eteil,
61 Avenue du G\'en\'eral de Gaulle, 94010 Cr\'eteil, France}

\email{stephane.sabourau@u-pec.fr}




\begin{abstract}
\forget
We study extremal systolic problems on Riemann surfaces, and
specifically Dyck's surface and its orientable genus-2 double cover.
There are three conformal types of genus-2 surfaces that are
significant for various problems of systolic extremality.  We will
denote them $A$, $B$, and $C$.  Here $A$ is the Bolza surface, $B$ is
a surface first considered in detail by Parlier, and $C$ is introduced
in the present text.  We show that (1) an extremal systolic inequality
for nonpositively curved metrics, whose case of equality is attained
by a quotient of $C$ by a fixedpointfree anticonformal involution,
equipped with a singular flat metric with finitely many conical
singularities; and (2) the conformal classes $A$, $B$, and~$C$ are
distinct.

In more detail, w
\forgotten

We prove an optimal systolic inequality for
nonpositively curved Dyck's surfaces.  The extremal surface is flat
with
%
%
eight conical singularities, six of angle~$\vartheta$ and two of
angle~$9 \pi -3 \vartheta$ for a suitable~$\vartheta$ with
$\cos(\vartheta)\in\Q(\sqrt{19})$.  Relying on some delicate capacity
estimates, we also show that the extremal surface is not conformally
equivalent to the hyperbolic Dyck's surface with maximal systole,
yielding a first example of systolic extremality with this behavior.
\end{abstract}

\maketitle

\tableofcontents

\section{Introduction}

\forget
The %
%
%
systole of a metric space~$M$ is the least length of a loop in~$M$
which cannot be contracted to a point.  

We will study extremal systolic problems on Riemann surfaces, and
specifically Dyck's surface and its orientable genus-2 double cover.
It turns out that there are three conformal types of genus-2 surfaces
that are significant for various problems of systolic extremality.  We
will denote them $A$, $B$, and $C$.  Here $A$ is the Bolza surface,
$B$ is a surface first considered in detail by Parlier, and $C$ is the
main subject of the current text.  Our main results are the
following:
\begin{enumerate}
\item
an extremal systolic inequality for nonpositively curved metrics,
whose case of equality is attained by a quotient of $C$ by a
fixedpointfree antiholomorphic involution, equipped with a singular
flat metric with finitely many conical singularities; and
\item
the conformal classes $A$, $B$, and~$C$ are distinct.
\end{enumerate}

\section{History and results}
\forgotten

The %
%
%
systole of a metric space~$M$ is the least length of a loop in~$M$
which cannot be contracted to a point.
Only a small number of optimal systolic inequalities relating the
systole and the volume of~$M$ are available in the literature.
Loewner's torus inequality asserts that every Riemannian torus~$\T$
satisfies
\begin{equation}
\label{11e}
\sys(\T^2)^2 \leq \frac{2}{\sqrt{3}} \, \area(\T^2);
\end{equation}
see Pu's paper \cite{Pu}.  Pu's inequality asserts that every real
projective plane~$\RP^2$ satisfies
\begin{equation}
\label{pu52}
\sys(\RP^2)^2 \leq \frac{\pi}{2} \, \area(\RP^2).
\end{equation}
Bavard's inequality ~\cite{Bav1} for the Klein bottle~$K$ is the bound
\begin{equation}
\label{11f}
\sys(K)^2 \leq \frac{\pi}{2\sqrt{2}} \, \area(K).
\end{equation}

The Burago-Ivanov-Gromov inequality relates the stable 1-systole of an
$n$-torus to its volume:
\begin{equation}
\label{Dgromsharp}
\stsys_1(\T^n)\le\sqrt{\gamma_n}\vol_n(\T^n)^{\frac{1}{n}},
\end{equation}
where~$\gamma_n$ is the Hermite constant; see Burago and Ivanov
\cite{BI94}, \cite{BI95}, Gromov \cite{Gr4}, and \cite[p.~155]{SGT}.

Bangert et al.{} proved the following optimal inequality for
orientable Riemannian~$n$-manifolds~$M$:
\begin{equation}
\label{ob2}
\stsys_1(M) \sys_{n-1}(M)\le \gamma'_b \vol_n(M)
\end{equation}
where~$b=b_1(M)$ is the first Betti number, and~$\gamma'_b$ is the
Berg\'e-Martinet constant; see Bangert et al. \cite{BK1}, \cite{BK2},
and \cite[p.~135]{SGT}.

In the context of nonpositively curved metrics, we proved the
following optimal inequality in~\cite{KS}: every nonpositively curved
genus two surface~$\Sigma_2$ satisfies the bound
\begin{equation} 
\label{eq:0S2}
\sys(\Sigma_2)^2 \leq \tfrac{1}{3} (\sqrt{2}+1) \, \area(\Sigma_2).
\end{equation}
Furthermore, the equality case is attained by a piecewise flat metric
with~$16$ conical singularities of angle~$\frac{9 \pi}{4}$.  Such a
metric can be expressed by a conformal factor vanishing at the
singular points, with respect to a smooth metric (e.g., the hyperbolic
one) in the conformal class of the Bolza surface.  The latter is the
surface possessing the maximal group of symmetries in genus 2.  The
\emph{hyperbolic} metric on the Bolza surface had also been shown to
have maximal systole in this genus by Jenni~\cite{Jen} and others.
The combination of these two (old) results can be summarized as
follows.

\begin{theorem}
\label{11}
The systolically extremal surface of nonpositive curvature and the
systolically extremal hyperbolic surface belong to the same conformal
class, namely that of the Bolza surface.
\end{theorem}

It turns out that such coincidence of conformal classes no longer
holds in the case of Dyck's surface (see Theorem~\ref{prop:distinct}).

In the present article, we develop an analogous optimal systolic
inequality for nonpositively curved metrics on Dyck's surface.  Here
Dyck's surface~$3\RP^2$ is the nonorientable closed surface of Euler
characteristic~$-1$, homeomorphic to the connected sums
\begin{eqnarray*}
3\RP^2 &=& \RP^2 \# \RP^2 \# \RP^2 \\&=& \RP^2 \# \T^2.
\end{eqnarray*}

By convention, a point on Dyck's surface will be referred to as a
Weierstrass point if it is the image of a Weierstrass point of the
orientable double cover.

\begin{theorem}
\label{theo:main1}
Each nonpositively curved metric on Dyck's surface~$3\rp^2$ satisfies
the following optimal systolic inequality:
\begin{equation} 
\label{eq:0D}
\sys(3\rp^2)^2 \leq \frac{12}{12+(169-38 \, \sqrt{19})^{\frac{1}{2}}}
\, \area(3\rp^2).
\end{equation}
Furthermore, the equality case is attained by a piecewise flat surface
composed of a flat M\"obius band and three identical symmetric
nonregular flat hexagons centered at its Weierstrass points.
\end{theorem}

Observe that the optimal constant~$\frac{1}{3} (\sqrt{2}+1) \simeq
0.80473$ in~\eqref{eq:0S2} is less than the one in~\eqref{eq:0D},
which is approximately~$0.86745$.  This is consistent with the
expected monotonicity of the systolic area as a function of the Euler
characteristic.  For general metrics, it is more difficult to obtain
estimates for the systolic ratio, but in \cite{KS11} we obtained a
small improvement (for Dyck's surface) of a general estimate of
Gromov's.

\forget
The extremal surface has eight conical singularities, which correspond
to the vertices of he hexagons: two of angle~$3(\pi-\theta)$ and six
of angle~$2 \pi + \theta$, where
\[
\theta = \arctan \sqrt{ \frac{8-\sqrt{19}}{2} } < \frac{\pi}{3}.
\]
The two conical singularities of angle~$3(\pi-\theta)$ correspond to
the points where the three hexagons meet, while those of angle~$2 \pi
+ \theta$ correspond to the points where the M\"obius band meets two
hexagons.  \forgotten

We will refer to the extremal surface as the \emph{extremal
nonpositively curved Dyck's surface} and denote it~$\DD_{\leq 0}$.  We
present two constructions of~$\DD_{\leq 0}$, one in
Section~\ref{sec:construction} and the other in
Section~\ref{sec:extremal}.  Theorem~\ref{theo:main1} is proved in
Section~\ref{six}.

The hyperbolic Dyck's surface of maximal systole, denoted
by~$\DD_{-1}$, was described by Parlier~\cite{Par} and then by
Gendulphe~\cite{Gen}.  It has the same symmetry group as the extremal
nonpositively curved Dyck's surface.  However, using some recent work
of B.~Muetzel~\cite{mue}, we show that they are not conformally
equivalent.

\begin{theorem}
\label{prop:distinct}
The conformal types of the extremal nonpositively curved Dyck's
surface~$\DD_{\leq 0}$ and the extremal hyperbolic Dyck's
surface~$\DD_{-1}$ are distinct, and can be distinguished by the
capacities of the associated annuli~$\mathcal{A}_{\leq 0}$
and~$\mathcal{A}_{-1}$ obtained by cutting open their orientable
double covers:
\[
\capa \mathcal{A}_{\leq 0} < 2.29 < \capa \mathcal{A}_{-1}
\]
(see Section~\ref{nine} for the definitions of the annuli).
\end{theorem}

This yields a first example of systolic extremality with this
behavior, in contrast with the situation in genus~$2$ summarized in
Theorem~\ref{11}.  In genus 3, the question is open.

\begin{question}
For genus~$3$ surfaces, do the extremal hyperbolic surface and the
extremal surface of nonpositive curvature belong to the same conformal
type?
\end{question}

We also observe that the orientable double cover of~$\DD_{\leq 0}$ is not
conformally equivalent to the Bolza surface;
\cf~Proposition~\ref{prop:bolza}.

Since the extremal hyperbolic metric and the extremal singular flat
metric on Dyck's surface are defined explicitly, one might have hoped
to find a (simpler) indirect proof of the fact that they belong to
distinct conformal classes.  However, this appears to be difficult.
We have therefore relied on the capacity estimate which produces a
direct numerical argument distinguishing between the two conformal
classes.

The proof of the main theorem makes use of an averaging argument by
the hyperelliptic involution and the Rauch comparison theorem as
in~\cite{KS}.  Unlike the situation treated in \cite{KS}, the
orientable double cover of the extremal surface does not have a
maximal group of symmetries.  Therefore our approach requires some new
features.  More specifically, it relies on the following:
\begin{enumerate}
\item
an analysis of the lengths of some distance curves to a canonical
nonorientable loop;
\item
an analysis of some optimal flat pieces (hexagons and cylinder) of a
partition of the extremal surface;
\item
a study of the combinatorics of these pieces;
\item
a delicate capacity estimate as explained above.
\end{enumerate}


\section{Description of the extremal surface} 
\label{sec:construction}

\forget
The extremal Dyck's surface~$\DD_{\leq 0}$ is flat with eight
singularities: six of angle~$\vartheta$ and two of
angle~$9\pi-3\vartheta$, with $\cos\vartheta\in\Q(\sqrt{19})$ (see
formula~\eqref{22}).  We describe below their positions in terms of
ramifications points of some ramified covers.

The orientable double cover of~$\DD_{\leq 0}$ is a genus two surface
which is a branched double cover over~$S^2$ with six branch points.
The branch points correspond to the vertices of an inscribed prism on
an equilateral triangle.

The Weierstrass points of the extremal surface are actually smooth,
therefore the quotient metric on~$S^2$ has conical singularities of
angle~$\pi$ at the six branch points, in addition to the singular
points arising from the extremal metric.

The singular points of the extremal metric on the orientable double
cover can be described as follows.  The centers of the triangular
faces of the prism lift to 4 conical singularities of
angle~$9\pi-3\vartheta$.  The other 12 conical singularities, of
angle~$\vartheta$, correspond to the vertices of another inscribed
prism on an equilateral triangle with the same axis as the branch
point prism.%
\footnote{\label{f3}Can this be explained in more detail?  What is
this other prism exactly?  It would be helpful to give a reference to
where it is described in the paper.}
\forgotten

\forget
in terms of the dual polyhedron of the branch point prism.
More precisely, the vertices of the dual polyhedron lift to conical
singularities of angle~$\vartheta$ and the centers of the two
triangular sides lift to conical singularities of
angle~$9\pi-3\vartheta$.%
\footnote{\label{f2}The centers of the triangular faces (not sides)
apparently are ALSO vertices of the dual polyhedron.  Perhaps this
should be clarified.}
\forgotten

\subsection{Construction of the extremal surface}

In addition to the description below, we will present an alternative
construction of the extremal surface at the end
Section~\ref{sec:extremal}.  We normalize the systole to~$1$.
Consider real numbers~$\alpha$ and~$h$ specified by the relations %
%
%
\begin{equation}
\label{18}
\left\{
\begin{array}{ccl}
\alpha & = & \frac{1}{2} \, \left( \pi - \arctan \sqrt{\frac{8-\sqrt{19}}{2}} \right) > \frac{\pi}{3} \\
 & & \\
h & = & \frac{1}{2} \, \cos \alpha \\
\end{array}
\right. 
\end{equation}

The piecewise flat surface defining the extremal nonpositively curved
Dyck's surface~$\DD_{\leq 0}$ can be constructed as follows.
\begin{enumerate}
\item
Take a flat isosceles trapezoid of height~$h$ and acute angle~$\alpha$
with the shorter (internal) side of length~$\frac{1}{3}$, where
parameters~$\alpha$ and~$h$ satisfy the relations~\eqref{18};
\item
Form a nonplanar hexagonal annulus~$\mathcal{H}$ composed of six
identical trapezoids (see~Figure~\ref{fig:0}), where the inner
boundary component of~$\mathcal{H}$ is of length~$2$;
\item
\label{three}
Form the torus with a disk removed, obtained from the hexagonal
annulus~$\mathcal{H}$ by identifying the opposite sides of the outer
boundary component of the annulus, as in Figure~\ref{fig:0};
\item
Attach a flat M\"obius band of width~$1-2h$ and boundary length~$2$ to
the torus with a disk removed constructed in step \eqref{three}.
\end{enumerate}

The values of the parameter $\alpha$ and~$h$ are chosen so that the contribution of the hexagonal annulus and the M\"obius band to the (systolic) area of the Dyck's surface is minimal.
This optimal configuration forces the systolic loops to decompose into three families, \cf~\S2.5.

\begin{figure}[
t] \setlength{\unitlength}{1pt}

\includegraphics[height=4cm]{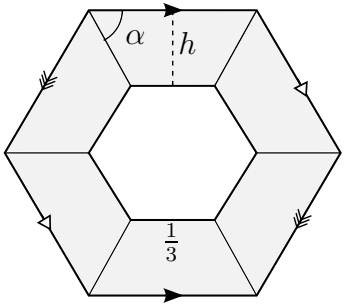}

\caption{The hexagonal annulus~$\mathcal{H}$} \label{fig:0}

\begin{picture}(0,0)(0,0)
\put(2,134){$h$}
\put(-18,138){$\alpha$}
\put(-4,60){$\frac{1}{3}$}
\end{picture}
\end{figure}

\subsection{Conical singularities}  

The extremal surface~$\DD_{\leq 0}$ so defined is of nonpositive
curvature in the sense of Alexandrov, since~$6\alpha>2\pi$.  More
precisely, it is piecewise flat and has eight conical singularities:
\begin{itemize}
\item
two of angle~$6\alpha$ corresponding to the vertices of the outer
boundary component of~$\mathcal{H}$, and 
\item
six of angle~$3 \pi - 2 \alpha$ corresponding to the vertices of the
inner boundary component of~$\mathcal{H}$. 
\end{itemize}
The angle~$\vartheta$ of six of the eight conical singularities
satisfies~$\vartheta = 3 \pi - 2 \alpha$, which leads to the following
expression:
\begin{equation}
\label{22}
\cos \vartheta = \frac{1+\sqrt{19}}{9} \in \Q(\sqrt{19}).
\end{equation}

\subsection{Weierstrass points}  The Weierstrass points
of~$\DD_{\leq 0}$ correspond to the midpoints of the sides of the
outer boundary component of~$\mathcal{H}$ after completion of the
steps (1) through (4) of the construction of the extremal surface.
The Weierstrass points of the extremal surface are actually smooth,
therefore the quotient metric on~$S^2$ from~\eqref{eq:CP1} has
conical singularities of angle~$\pi$ at the six branch points, in
addition to the singular points arising from the extremal metric.

The optimal configuration for~$\DD_{\leq 0}$ follows from an
equilibrium between the systolic area contribution of the M\"obius
band and that of the Voronoi cells centered at the Weierstrass points.

\subsection{Automorphism groups}  
\label{24}


The following proposition provides a description of the automorphism and symmetry groups of~$\DD_{\leq0}$ viewed as a Riemannian Klein surface.

\begin{proposition} \label{prop:sym}
The automorphism group and the symmetry group of~$\DD_{\leq0}$ are
both isomorphic to
\[
D_3 \times \Z/2\Z.
\]
\end{proposition}

\begin{proof}
The natural homomorphism between the automorphism (resp. symmetry)
group of~$\DD_{\leq0}$
and the permutation group~$D_3$ of its Weierstrass points is
surjective.  Its kernel is composed of dianalytic (\ie locally
holomorphic or antiholomorphic) automorphisms preserving the
Weierstrass points.  The only two automorphisms with this property are
the identity map and the hyperelliptic involution (which is an
isometry).  Since they commute with every holomorphic map (and so
every isometry), \cf~\cite[\S III.9]{FK},
we deduce that the automorphism (resp. symmetry) group
of~$\DD_{\leq0}$ is isomorphic to
$
D_3 \times \Z/2\Z
$.
\end{proof}

\forget
More specifically, it can be described as follows.  
The orientable double cover of~$\DD_{\leq 0}$ is a genus two Riemann
surface~$\Sigma_2$ which can be thought of as a holomorphic ramified
double cover
\begin{equation} 
\label{eq:cover}
\Sigma_2 \to \CP^1 \simeq S^2
\end{equation}
whose branch points (\ie the images of the Weierstrass points
of~$\Sigma_2$) form the set of~$6$ vertices of a triangular prism
inscribed in~$S^2$.  An affine form of such a Riemann surface can be
taken to be the locus in~$\C^2$ of the equation
\[
y^2 = x^6-ax^3+1
\]
for a suitable~$a>0$ (cf.~equation~\eqref{eq:y2=} below).

The prism can be taken to be a standard right prism on an equilateral
triangle inscribed in a unit sphere in~$\R^3$.  The symmetry group
$D_3\times \Z/2\Z$ of this prism can be described geometrically as
follows.  The dihedral group~$D_3$ acts by 
%
%
isometries on the prism (globally) preserving each triangular face.
%
%
The second factor~$\Z/2\Z$ is the center of the symmetry group,
and is generated by an orientation-reversing involution given by the
reflection in the plane (through the origin) parallel to the
triangular faces.%
%
%

The double cover~$\Sigma_2$ has an additional automorphism given by
the hyperelliptic involution.  Its group of holomorphic and
antiholomorphic automorphisms is of order 24 and is isomorphic
to~$D_3\times(\Z/2\Z)^2$.  The surface~$\Sigma_2$ belongs to the
complex family of Riemann surfaces
%
%
with holomorphic automorphism group~$D_3 \times \Z/2\Z$ given by
Bolza's classification~\cite{bol} (see also the presentation
in~\cite{kkk}).  

Note that the extremal nonpositively curved genus two surface admits a
description in terms of a uniform \emph{antiprism} (rather than a
prism), namely, a regular octahedron (\cf~\cite{KS}).  Its holomorphic
automorphism group is isomorphic to~$S_4\times\Z/2\Z$.  The unique
Riemann surface with this automorphism group is Bolza's curve,
\cf~\cite{bol}, \cite{kkk}.


In particular, the conformal class of the double cover of the
surface~$\DD_{\leq 0}$ is distinct from the Bolza curve.
\forgotten

\subsection{Systolic loops} \label{subsec:loops}

There are three types of systolic loops on the extremal surface:
\begin{itemize}
\item
the soul of the flat M\"obius band,
\item 
the loops orthogonal to a short base of one of the trapezoids
of~$\mathcal{H}$,
\item
the loops orthogonal to a leg %
%
%
of one of the trapezoids of~$\mathcal{H}$. 
\end{itemize}
Note that the last type of systolic loops contains loops joining any
pair of Weierstrass points.

\subsection{Extremality}  Since the extremal surface admits
regions at every point of which exactly one systolic loop passes, it
is not extremal for the curvature-free systolic inequality
on~$3\RP^2$; \cf~\cite[Lemma~2.1]{cal}.  Actually, a direct
application of the characterization of conformally extremal surfaces
established by Bavard in~\cite{bav92} shows that~$\DD_{\leq 0}$ is not
even extremal in its conformal class for the curvature-free systolic
inequality.  Thus, an extremal metric for the curvature-free systolic
inequality on~$3\RP^2$ necessarily admits regions of both positive and
negative curvature.

\forget
We also describe the Riemann surface structure of the extremal
nonpositively curved Dyck's surface.

\begin{theorem} 
\label{theo:main2}
The algebraic curve representing the orientable double cover of the
extremal nonpositively curved Dyck's surface is the smooth completion
of the following affine curve in~$\C^2$:
\begin{equation}
\label{18}
y^2 = (x^2+1)(x^4+1).
\end{equation}
\end{theorem}

The orientable double cover of the hyperbolic Dyck's surface~$3\rp^2$
of maximal systole was described by Parlier~\cite{Par}.
Silhol~\cite{Si, Si10} identified a presentation of its affine form:
\begin{equation}
\label{19}
y^2=x^6+ ax^3+1,
\end{equation}
where~$a=434+108 \sqrt{17}$.  See also Leli\`evre \& Silhol \cite{LS}
and Gendulphe~\cite{Gen}.

The set of the roots of the equations \eqref{18} and \eqref{19} are
{\em not\/} equivalent by a Mobius transformation of~$\hat \C$.
Indeed, the roots of~\eqref{18} lie on a common circle, while those
of~\eqref{19} lie neither on a common circle, nor on a straight line
(a similar argument shows that the algebraic curve~\eqref{18} is not
conformally equivalent to the Bolza surface of equation~$y^2=x^5-x$).
We will also detect differences at the level of the hyperbolic metrics
in Section~\ref{last}.  We therefore obtain the following corollary.
\forgotten

\section{Conformal data} 
\label{sec:conf}

In this section, we review some conformal constructions and results
that we will need for the proofs of Theorem~\ref{theo:main1} and
Theorem~\ref{prop:distinct}.  A Riemann surface homeomorphic
to~$3\RP^2$ is the quotient~$3\RP^2=\Sigma_2/\tau$ of a genus two
Riemann surface~$\Sigma_2$ by a fixed point-free antiholomorphic
involution~$\tau$.  Recall that every genus two Riemann
surface~$\Sigma_2$ is hyperelliptic.  It admits an affine model
\begin{equation} 
\label{eq:y2=}
y^2 = p(x)
\end{equation}
in~$\C^2$, where~$p$ is a degree~$6$ complex polynomial with six
distinct roots, which correspond to the Weierstrass points
of~$\Sigma_2$.  In this presentation, the hyperelliptic involution
\[
J: \Sigma_2 \to \Sigma_2
\]
is given by the transformation~$(x,y) \mapsto (x,-y)$ on the affine
model.  It is the only holomorphic involution of~$\Sigma_2$ with six
fixed points (which are the six Weierstrass points).  By uniqueness, %
%
%
every holomorphic or antiholomorphic involution of~$\Sigma_2$ commutes
with~$J$; \cf~\cite[\S III.9]{FK}.  In particular,
\begin{equation} 
\label{eq:unique}
\tau \circ J = J \circ \tau.
\end{equation}
Thus, the hyperelliptic involution~$J$ on~$\Sigma_2$ descends to an
involution, denoted~$J_{\mathcal{D}}$, on~$3\RP^2$:
\begin{equation} 
\label{eq:JJ}
J_{\mathcal{D}}: 3\RP^2 \to 3\RP^2.
\end{equation}

The projection of the locus of the equation~\eqref{eq:y2=} to
the~$x$-coordinate induces a holomorphic double cover
\begin{equation} 
\label{eq:CP1}
Q: \Sigma_2 \to \CP^1
\end{equation}
ramified at the Weierstrass points of~$\Sigma_2$.  The presence of the
real structure~$\tau$ entails that the polynomial~$p$
in~\eqref{eq:y2=} may be assumed to have real coefficients, and that
the involution~$\tau:\Sigma_2 \to \Sigma_2$ restricts to the complex
conjugation on the affine model, namely
\[
\tau(x,y) = (\bar{x},\bar{y}).
\]
Since the upperhalf plane in~$\C$ is a fundamental domain for the
action of the complex conjugation, the points on~$3\RP^2=
\Sigma_2/\tau$ can be represented by points in the closure of the
upperhalf plane.  More precisely, consider the northern hemisphere
\[
\hat\C^+\subset \CP^1 = \C \cup \{\infty\}=S^2,
\]
with the equator included. 
We will think of the surface~$3\RP^2$ as the ramified double cover
\begin{equation}
\label{dc}
3\RP^2\to \hat\C^+
\end{equation}
induced by~\eqref{eq:CP1}.  Such a cover is branched along the equator
of the hemisphere~$\hat\C^+$ as well as at three additional branch
points, namely the projections to~$\hat\C^+$ of the Weierstrass points
of~$\Sigma_2$.

\begin{definition}
\label{31}
Let~$Y\subset \Sigma_2$ be the preimage of the northern
hemisphere~$\hat\C^+$ under the double cover~$Q$ of~\eqref{eq:CP1}.
\end{definition}

Clearly,~$Y$ is a torus with an open disk removed.

Let~$\tau_\partial$ be the restriction of the involution~$\tau$
of~$\Sigma_2$ to the boundary circle~$\partial Y$.  The map
$\tau_\partial$ is the antipodal map on the boundary circle~$\partial
Y$.  The original surface~$3\rp^2$ can be viewed as the quotient space
of~$Y$ by~$\tau_\partial$:
\begin{equation} 
\label{eq:Yt}
Y \to 3\RP^2=Y/\tau_\partial^{\phantom{I}}.
\end{equation}
By construction, the subsurface~$Y\subset \Sigma_2$ is~$J$-invariant
and its boundary~$\partial Y$ is the fixed point set of the
involution~$J \circ \tau$ of \eqref{eq:unique} on~$\Sigma_2$.

\section{Area lower bound for some collars} 

We will reduce the problem to symmetric metrics, and then establish an
optimal lower bound for the area of some collars in~$3\RP^2$.  Similar
considerations for general metrics, albeit with worse constants, can
be found in~\cite{KS11}.

\forget
In this section, we begin the proof of the sharp systolic inequality
of Theorem~\ref{theo:main1} by decomposing nonpositively curved Dyck's
surfaces into various regions.  The rest of the proof will occupy the
next two sections.
\forgotten

\begin{definition}
The average metric~$\bar{g}$ of a Riemannian metric~$g$ on~$3\RP^2$ by
the hyperelliptic involution~\eqref{eq:JJ} is defined as
$$ \bar{g} = \frac{g+J^*_{\mathcal{D}}(g)}{2}.
$$
\end{definition}

We will need the following results regarding the average metric.

\begin{lemma}[\cite{KS}, Lemma~4.2]
\label{lem:1}
The average metric~$\bar{g}$ of a nonpositively curved Riemannian
metric~$g$ on~$3\RP^2$ is similarly nonpositively curved.
\end{lemma}

\begin{lemma}[\cite{SGT}]
\label{lem:2}
Let~$g$ be a Riemannian metric on~$3\RP^2$.  The average
metric~$\bar{g}$ has a better systolic area than~$g$, that is,
$$
\frac{\area(\bar{g})}{\sys(\bar{g})^2} \leq \frac{\area(g)}{\sys(g)^2}.
$$
\end{lemma}

\forget
\begin{theorem}
Every nonpositively curved metric on Dyck's surface~$3\rp^2$ satisfies
\[
\area(3\rp^2) \geq \sys(3\rp^2)^2.
\]
Furthermore, the equality case is attained by a nonpositively curved
(in the sense of Alexandrov) piecewise flat metric with one conical
singularity of angle~$4\pi$.
\end{theorem}
\forgotten

\bigskip

By Lemmas~\ref{lem:1} and~\ref{lem:2}, the nonpositively curved metric
on~$3\RP^2$ (and so its lift to~$\Sigma=\Sigma_2$) may be assumed
invariant under the hyperelliptic involution.  We normalize the metric
by rescaling it to unit systole, so that~$\sys(3\RP^2)=1$.  These
assumptions on the metric will be implicit throughout the article. \\

Consider the surface~$Y \subset \Sigma$ as in Definition~\ref{31}.
Recall that~$Y$ is a torus with an open disk removed.  Here,~$Y$ is
endowed with the~$J$- and~$\tau$-invariant nonpositively curved metric
inherited from~$\Sigma$.  As the fixed-point set of the isometric
involution~$J \circ \tau$, the boundary~$\partial Y$ of~$Y$ is
geodesic.

\begin{lemma}
\label{lem:level}
Relative to the normalisation~$\sys(3\RP^2)=1$, the level curves
of~$Y$ at distance less than~$\frac{1}{2}$ from~$\partial Y$ are loops
freely homotopic to~$\partial Y$.  Furthermore, they are of length at
least~$2$.
\end{lemma}

\begin{proof}
By Morse theory, the level curve at distance~$r$ from~$\partial Y$
deforms to~$\partial Y$ if the function~$f(p)=\text{dist}(p,\partial
Y)$ has no singular value between~$0$ and~$r$.  Let~$r$ be the least
value for which this is not the case.

Since~$Y$ is nonpositively curved, there exist two length-minimizing
paths of length~$r$ joining~$\partial Y$ to the same critical point
of~$f$ on~$f^{-1}(r)$.  Furthermore, these two length-minimizing paths
form a geodesic arc~$\gamma$ with endpoints in~$\partial Y$ which
induces a nontrivial class in~$\pi_1(Y,\partial Y)$.
Note that~$\gamma$ orthogonally meets~$\partial Y$ at its endpoints.

Now, the hyperelliptic involution~$J$ on~$Y$ induces the
homomorphism~$-id$ on~$\pi_1(Y,\partial Y)$.  Thus, the arcs~$\gamma$
and~$-J\gamma$ lie in the same relative homotopy class
in~$\pi_1(Y,\partial Y)$.  From the flat strip theorem, these two
geodesic arcs are parallel and bound with some arcs of~$\partial Y$ %
%
%
a~$J$-invariant flat rectangle.  The center~$x$ of this rectangle is
clearly a Weierstrass point and
\begin{equation} \label{eq:Lgamma}
\length \, \gamma = 2r = 2 \, {\rm dist}(x,\partial Y).
\end{equation}

Now, the segment joining~$\partial Y$ to~$x$ forms with its image
by~$J$ a geodesic arc~$c$ with opposite endpoints on~$\partial Y$.
This arc~$c$ projects to a noncontractible loop of~$3\RP^2$.  Hence
the length of~$c$, which is twice~${\rm dist}(x,\partial Y)$, is at
least~$1$.  Combined with~\eqref{eq:Lgamma}, we derive the first part
of the lemma, namely
\[
r \geq \frac{1}{2}.
\]

Each of the two arcs of~$\partial Y$ joining a pair of opposite points
projects to a noncontractible loop of~$3\rp^2$.  Therefore, the length
of~$\partial Y$ is at least~$2$.  Since~$Y$ is nonpositively curved
and~$\partial Y$ is a closed geodesic, every loop of~$Y$ freely
homotopic to~$\partial Y$ is of length at least~$2$, and so are the
level curves of~$Y$ at distance less than~$r$ from~$\partial Y$.
\end{proof}

The coarea inequality yields a first lower bound on the systolic area
of a nonpositively curved metric on~$3\RP^2$.

\begin{proposition} 
\label{prop:U}
Consider a normalized nonpositively curved metric
on~$3\RP^2=Y/\tau_\partial^{\phantom{I}}$ invariant by the
hyperelliptic involution and let~$\delta \in (0, \frac{1}{2})$.  The
$\delta$-tubular neighborhood~$U_\delta$ of~$\partial Y$ in~$Y$ is a
topological cylinder which satisfies
\begin{equation*}
\area(U_\delta) \geq 2\delta.
\end{equation*}
In particular,
\begin{equation} \label{eq:first}
\sys(3 \RP^2)^2 \leq \area(3\RP^2).
\end{equation}
\end{proposition}

\section{Decomposition of nonpositively curved Dyck's surfaces}
\label{sec:decomp}

We introduce a decomposition of Dyck's surface leading to the
description of the extremal metric in Section~\ref{sec:extremal}.
This decomposition will also allow us to estimate the systolic area of
nonpositively curved metrics.

The decomposition is motivated by the following observation.  If the
systolic inequality~\eqref{eq:first} were optimal, the extremal
surface could be defined from a flat cylinder of circumference~$2$ and
height~$\frac{1}{2}$ by identifying pairs of opposite points on one of
its boundary component (which leads to a M\"obius band) and by gluing
the other boundary component onto itself so as to obtain the right
topological type for the surface.  It turns out this is impossible
without decreasing the systole, which shows that the
inequality~\eqref{eq:first} is not optimal.

Actually, we will see that the extremal surface decomposes into a flat 
M\"obius band and a torus with a disk removed.
The flat M\"obius band is defined as previously from a flat cylinder of
circumference~$2$ and height~$\delta < \frac{1}{2}$, while the torus
with a disk removed is made of three isometric flat hexagons. 

The conical singularities will correspond to the points where the
various flat regions meet.  The sizes of the cylinder and the hexagons
must be chosen to make the gluing possible while minimizing the
systolic area.  The comparison with nonpositively curved metrics will
be carried out afterwards.

Let us introduce some quantities related to the sizes of the cylinder
and hexagons, whose geometric interpretations will appear in
Proposition~\ref{prop:hexa}.

Fix~$h>0$ and~$\theta \in (0,\frac{\pi}{2})$ such that
\begin{equation} 
\label{eq}
\left\{
\begin{array}{ccc}
2h & = & \sin \frac{\theta}{2} \\
& & \\
6h & = & \tan \theta
\end{array}
\right.
\end{equation}
Note that
\begin{equation} \label{eq:tan}
\tan \frac{\theta}{2} = \frac{2h}{\sqrt{1-4 h^2}}.
\end{equation}
More explicitly, we have
\begin{equation}
\label{eq:h}
\left\{
\begin{array}{ccc}
h & = & \sqrt{\frac{8-\sqrt{19}}{72}} \,  \simeq \, 0.2248796 \\
 & & \\
\theta & = & \arctan \sqrt{\frac{8-\sqrt{19}}{2}} < \frac{\pi}{3}
\end{array}
\right.
\end{equation}
Let~$\delta \in (0,\frac{1}{2})$ be defined as~$\delta=\frac{1}{2}-h$.

We will use the notations and assumptions from the previous section.
In particular, we assume that~$3\RP^2=Y/\tau_\partial^{\phantom{I}}$
as in~\eqref{eq:Yt} is endowed with a normalized nonpositively curved
metric invariant under the hyperelliptic involution~\eqref{eq:JJ}.  We
would like to decompose~$Y$ into~$U_{\delta}$ and three Voronoi cells
centered at the three Weierstrass points of~$Y$.  The main theorem
will then follow from a comparison between the areas of these Voronoi
cells and those of some Euclidean polygons.  In order to describe the
Voronoi cells and their comparison Euclidean polygons, it is
convenient to proceed as follows.

Since~$\Sigma$ is nonpositively curved, the open collar~$C_{\delta}$
of width~$\delta$ around the closed geodesic~$\partial Y$ of~$\Sigma$
is convex.  Removing this collar and gluing back the boundary
components of~$\Sigma \setminus C_{\delta}$ yields a new~${\rm
CAT}(0)$ genus two surface~$\Sigma_0$.  We will identify the regions
of~$\Sigma_0$ with those of~$\Sigma \setminus C_{\delta}$.

Let~$\tilde{\Sigma}_0$ be the universal cover of~$\Sigma_0$.  The
Voronoi cell of~$\tilde{\Sigma}_0$ centered at a lift of a Weierstrass
point of~$\Sigma_0$ is the region of~$\tilde{\Sigma}_0$ formed of the
points closer to this lift than to any other lift of a Weierstrass
point.  The Voronoi cells on~$\tilde{\Sigma}_0$ are polygons whose
edges are arcs of the equidistant curves between a pair of lifts of
Weierstrass points.  Note that these edges are not necessarily
geodesics.  Since the metric is nonpositively curved, the Voronoi
cells on~$\tilde{\Sigma}_0$ are topological disks, while their
projections to~$\Sigma$ or~$3\RP^2$, still called Voronoi cells, may
have more complicated topology.
Note that the Voronoi cells on~$\tilde{\Sigma}_0$ have the same area as their
projections to~$\Sigma$ or~$3\RP^2$.
This is because the \emph{interior} of a Voronoi cell projects injectively.

\forget
\begin{proposition}
The Voronoi cells on~$\tilde{\Sigma}_0$ have the same area as their
projections to~$\Sigma$ or~$3\RP^2$.
\end{proposition}

\begin{proof}
This is because the \emph{interior} of a Voronoi cell projects injectively.
\end{proof}
\forgotten

Furthermore, since the metric on~$\Sigma_0$ is~$J$-invariant, the
Voronoi cells on~$\tilde{\Sigma}_0$ are symmetric with respect to
their centers, and since it is also~$(J \circ \tau)$-invariant, the
projections of their boundaries to~$\Sigma$ contain the boundary
components of~$C_{\delta}$. \\

Now assume~$y \in \tilde{\Sigma}_0$ is the lift of a Weierstrass
point, and let~$y' \in \tilde{\Sigma}_0$ be the lift of another
Weierstrass point.  Denote by~$x'\in T_y$ the preimage of~$y'$ by the
exponential map~${\rm exp}_y$ from the tangent plane~$T_y$
to~$\tilde{\Sigma}_0$ at~$y$.  Define~$L_{x'}$ as the equidistant line
in~$T_y$ between the origin~$x \in T_y$ and~$x'$.  For every~$x'' \in
L_{x'}$, set~$y'' = {\rm exp}_y(x'')$.  By the Rauch comparison
theorem, the exponential map~${\rm exp}_y$ does not decrease
distances.  Thus,
\begin{equation} 
\label{eq:rauch}
\dist_{\tilde{\Sigma}_0}(y,y'')=\dist_{T_y}(x,x'')=
\dist_{T_y}(x',x'')\leq\dist_{\tilde{\Sigma}_0}(y',y'').
\end{equation}

Consider the Euclidean polygon in the tangent plane~$T_y$, obtained as
the intersection of the halfspaces containing the origin, defined by
the lines~$L_{x'}$, as~$y'$ runs over all Weierstrass points distinct
from~$y$.  This Euclidean polygon will be referred to as the
\emph{comparison Euclidean polygon} corresponding to the Voronoi cell
centered at~$y$.  It follows from the inequality~\eqref{eq:rauch} that
the exponential image of this polygon is contained in the Voronoi cell
of~$y$.  Since the exponential map does not decrease distances, we
obtain the following proposition.

\begin{proposition} \label{prop:voronoi}
The area of a Voronoi cell is bounded from below by the area of its
comparison Euclidean polygon.
\end{proposition}

By construction,~$Y$ decomposes into~$U_{\delta}$ and three Voronoi
cells centered at the three Weierstrass points of~$Y$. \\

We conclude this section with some distance estimates on the centers
of the Voronoi cells of~$Y$, that is, on the Weierstrass points.  

\begin{lemma} \label{lem:dist}
\mbox{ }
\begin{enumerate}
\item The distance between two Weierstrass points of~$Y$ is at least~$\frac{1}{2}$. 
\item Every Weierstrass point of~$Y$ is at distance at least~$h$ from~$U_\delta$.
\end{enumerate}
\end{lemma}

\begin{proof}
Every minimizing segment between a pair of isolated branch points of the
double cover~$3\RP^2 \to \hat{\C}^+$ lifts to a noncontractible loop
of~$3\RP^2$.  Thus, the distance between two Weierstrass points of~$Y$
is at least~$\frac{1}{2}$.  Similarly, every minimizing segment
between an isolated branch point of~$3\RP^2 \to \hat{\C}^+$ and the
equator lifts to a noncontractible loop of~$3\RP^2$.  Thus, every
Weierstrass point of~$Y$ is at distance at least~$\frac{1}{2}$
from~$\partial Y$ and so at distance at least~$h$ from~$U_\delta$.
\end{proof}

%
%

\section{The extremal nonpositively curved Dyck's surface} 
\label{sec:extremal}

In this section, we bound from below the area of the Voronoi cells in
some special case and describe the extremal nonpositively curved
Dyck's surface using the constructions and notation defined earlier.

As previously, we assume that~$3\RP^2$ is endowed with a normalized
nonpositively curved metric invariant under the hyperelliptic
involution.  This metric descends to a singular metric on~$\hat{\C}^+$
under the ramified cover~\eqref{dc}.  

\begin{definition}
\label{51}
Let~$\Gamma$ on~$\hat{\C}^+$ be the connected graph given by the
projections of the edges of the Voronoi cells to the hemisphere,
see~Proposition~\ref{prop:voronoi}.  Denote by~$\hat{\C}^+_\delta$ the
spherical cap of~$\hat{\C}^+$ bounded by the level curve~$c_\delta$
of~$\hat{\C}^+$ at distance~$\delta$ from the equator.  From
Lemma~\ref{lem:level}, the length of~$c_\delta$ is at least~$1$.
\end{definition}

By construction, the graph~$\Gamma$ lies in~$\hat{\C}^+_\delta$,
contains~$c_\delta$ and bounds~$f=3$ faces.  From the
formula~$v-e+f=1$, where~$v$ and~$e$ are the numbers of vertices and
edges of~$\Gamma$, and the well-known inequality~$3v \leq 2e$, we
derive that~$\Gamma$ has at most~$6$ edges and~$4$ vertices.

Suppose that~$\Gamma$ has three vertices lying in~$c_\delta$ and a
fourth one in the interior of~$\hat{\C}^+_\delta$ from which arise
three edges connecting the three other vertices on~$c_\delta$,
\cf~Figure~\ref{fig:a}.  In other words,~$\Gamma$ bounds three
triangles in~$\hat{\C}^+_\delta$.

\begin{figure}[ht]
\setlength{\unitlength}{1pt}

\begin{picture}(0,0)(0,0)
\put(35,-100){$\hat{\C}^+$}
\put(23,-58){$c_\delta$}
\end{picture}

\includegraphics[height=4cm]{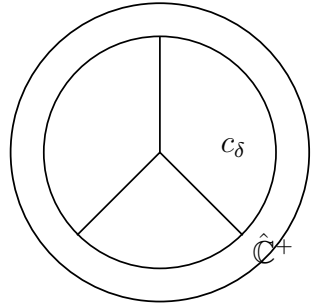}

\caption{Triangular decomposition of~$\hat{\C}^+_\delta$} \label{fig:a}
\end{figure}

By construction, each of these triangles lifts to a hexagonal Voronoi
cell in~$\tilde{\Sigma}_0$ whose comparison Euclidean polygon is a
(symmetric) hexagon.  Furthermore, the center of this Euclidean
hexagon~$H$ is at distance at least~$\frac{1}{4}$ from two pairs of
opposite sides and at distance at least~$h$ from the other pair of
opposite sides, see Lemma~\ref{lem:dist} \\

The following result provides a sharp lower bound on the area of the
comparison Euclidean hexagon and therefore on the hexagonal Voronoi
cells.

\begin{proposition}
\label{prop:hexa}
Let~$H$ be the symmetric Euclidean hexagon which is the comparison
hexagon of the Voronoi cell as above. Then
\[
\area(H)\geq h\,\sqrt{1-4h^2}.
\]
Furthermore, the equality case is attained by a symmetric Euclidean
hexagon composed of six pairwise opposite isosceles triangles based at
its center: four of them have height~$\frac{1}{4}$ and main
angle~$\theta$, and two of them have height~$h$ and
base~$\frac{1}{3}$, \cf~Figure~\ref{fig:h}.  Here,~$h$ and~$\theta$
are defined in~\eqref{eq}.
\end{proposition}

\begin{figure}[ht]
\setlength{\unitlength}{1pt}

\begin{picture}(0,0)(0,0)

\put(57,-32){$a_1$}
\put(-75,-85){$a_1'$}

\put(-9,1){$a_2$}
\put(-7,-120){$\frac{1}{3}$}

\put(-75,-32){$a_3$}
\put(57,-85){$a_3'$}

\put(-14,-28){$h$}
\put(37,-49){$\frac{1}{4}$}
\put(13,-53){$\theta$}
\put(-9,-72){$O$}

\end{picture}

\includegraphics[height=4cm]{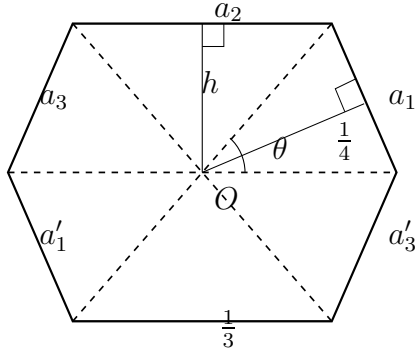}

\caption{The hexagon~$H$} \label{fig:h}
\end{figure}

\begin{remark} \label{rem:hexa}
Note that the angles between two sides of length other
than~$\frac{1}{3}$ are equal to~$\pi-\theta>\frac{2\pi}{3}$.
\end{remark}

\begin{proof}[Proof of Proposition~\ref{prop:hexa}]
Let~$O$ be the center of~$H$.  Denote by~$a_1$,~$a_2$, 
$a_3$,~$a'_1$,~$a'_2$ and~$a'_3$ the sides of~$H$ (in this order)
with~$a'_i$ opposite to~$a_i$ such that \mbox{$d(O,a_1) \geq
\frac{1}{4}$},~$d(O,a_2) \geq h$ and~$d(O,a_3) \geq \frac{1}{4}$,
\cf~Figure~\ref{fig:h}.  The area of the triangle~$T_i$ with
vertex~$O$ and side~$a_i$ is bounded from below by
\[
d(O,a_i)^2 \, \tan \frac{\alpha_i}{2},
\]
where~$\alpha_i$ is the angle of~$T_i$ at~$O$.
Thus, from the relation \mbox{$\alpha_1 + \alpha_2 + \alpha_3 = \pi$}, we have
\begin{eqnarray*}
\area H & \geq & 2 \, \left( \frac{1}{4} \right)^2 \, \tan \frac{\alpha_1}{2} + 2 \, h^2 \, \tan \frac{\alpha_2}{2} + 2 \, \left( \frac{1}{4} \right)^2 \, \tan \frac{\alpha_3}{2} \\
 & \geq & 4 \, \left( \frac{1}{4} \right)^2 \, \tan \frac{\alpha}{2} + 2 \, \frac{h^2}{\tan \alpha},
\end{eqnarray*}
where~$\alpha = \frac{\alpha_1 + \alpha_3}{2}$.
The second inequality comes from the definition of convexity applied to the convex function~$\tan (x/2)$ between~$\alpha_1$ and~$\alpha_3$.
Using a classical relation between~$\tan(\alpha)$ and~$\tan (\alpha/2)$, we observe that this lower bound is minimal when
\[
\tan^2 \frac{\alpha}{2} = \frac{4h^2}{1-4h^2},
\]
that is, when~$\alpha = \theta$ from~\eqref{eq:tan}.
The minimal lower bound is
\begin{equation} \label{eq:area}
h \, \sqrt{1-4h^2}.
\end{equation}
Furthermore, the equality case occurs only if~$T_1$ and~$T_3$ are isosceles triangles of height~$\frac{1}{2}$ and main angle~$\theta$, and~$T_2$ is an isosceles triangle of height~$h$ and main angle~$\pi-2\theta$.
From our choice of~$h$ and~$\theta$, \cf~\eqref{eq}, the sides arising from the main vertices of these isosceles triangles  have the same length, namely 
\[
\frac{1/4}{ \cos \frac{\theta}{2}} = \frac{h}{\sin \theta}.
\]
This shows that it is possible to put these isosceles triangles together in order to obtain an hexagon satisfying the desired constraints with area~\eqref{eq:area}.
In this case, the base of~$T_2$ is of length
\[
2 \, \frac{h}{\tan \theta} = \frac{1}{3}.
\]
\end{proof}

From~Proposition~\ref{prop:U} and Proposition~\ref{prop:hexa}, the
nonpositively curved Dyck's surface~$3\RP^2$, which decomposes into
the~$\delta$-tubular neighborhood~$U_\delta$ and three hexagonal
Voronoi cells, satisfies the following area lower bound
\begin{eqnarray}
\area(3\RP^2) & \geq & 2\delta + 3h \, \sqrt{1-4h^2} \nonumber \\
 & \geq & 1 + \frac{1}{12} \, \sqrt{169 - 38 \, \sqrt{19}} \label{eq:extremal}
\end{eqnarray}

This area lower bound is optimal.  It is attained by the nonpositively
curved (in Alexandrov's sense) piecewise flat Dyck's
surface~$\DD_{\leq0}$ obtained as follows.  Glue three copies of the
optimal flat hexagon described in Proposition~\ref{prop:hexa} and
identify the opposite sides of lengths other than~$\frac{1}{3}$,
\cf~Figure~\ref{fig:b} and Remark~\ref{rem:hexa}. \\

\begin{figure}[ht]

\begin{picture}(0,0)(0,0)

\put(59,-61){$q$}
\put(-45,-1){$q$}
\put(-45,-123){$q$}

\put(-14,-71){$p$}

\end{picture}

\includegraphics[height=4cm]{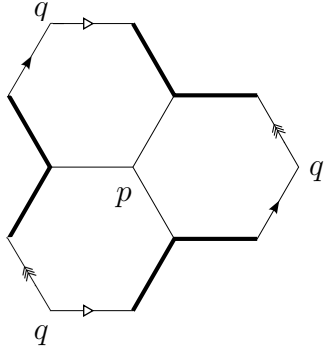}
\caption{Three hexagonal Voronoi cells} \label{fig:b}
\end{figure}

The resulting surface is of boundary length~$2$.  Now attach a flat
cylinder of circumference~$2$ and height~$\delta$ to it and identify
the opposite boundary points.  We obtain a nonpositively curved (in
Alexandrov's sense) piecewise flat Dyck's surface~$\DD_{\leq0}$ with
two conical singularities~$p$ and~$q$ of angle~$3 (\pi-\theta) > 2\pi$
where the three hexagons meet, and six conical singularities of
angle~$2\pi+\theta$ at the other vertices of the hexagons.

Its area satisfies
\begin{equation}
\label{eq:A}
\area \DD_{\leq0} = 1 + \frac{1}{12} \, \sqrt{169 - 38 \, \sqrt{19}}
\, \simeq \, 1.15279.
\end{equation}

\begin{proposition}
The surface~$\DD_{\leq 0}$ has unit systole.
\end{proposition}

\begin{proof}
\forget
Consider one of the two singular points (corresponding to the ``north
pole").  Consider its cut locus ``in the tangent plane".  This is not
actually a plane but a cone with total angle~$2\pi-3\theta$.  The cut
locus should be something like a regular hexagon with the six singular
points lying on its boundary.

It seems to me that the cut locus can be localized on this cone
without any additional singularities (other than those on the
boundary).  This would help clarify the picture at least as far as the
injectivity radius from this particular point is concerned, and
clarify the picture with the systole of the extremal surface.
\forgotten

By construction, the extremal surface~$\DD_{\leq 0}$ decomposes into
an open flat M\"obius band~$\mathcal{M}$ and six trapezoids;
\cf~Section~\ref{sec:construction}.  The M\"obius band has unit
systole.  Hence, the systole of~$\DD_{\leq 0}$ is at most~$1$.
Furthermore, every geodesic arc of~$\mathcal{M}$ with endpoints
on~$\partial \mathcal{M}$ is of length at least~$2\delta=1-2h$.
Similarly, every geodesic arc of the complement
\[
\DD_{\leq 0} \setminus \mathcal{M}
\]
with endpoints on~$\partial \mathcal{M}$ is of length at least~$2h$.
Thus, a noncontractible geodesic loop of~$\DD_{\leq 0}$
intersecting~$\mathcal{M}$ is of length at least~$1$.

Consider now a systolic loop~$\gamma$ of~$\DD_{\leq 0}$ which does not
meet~$\mathcal{M}$.  Denote by~$c$ its projection to~$\hat{\C}^+$ by
the ramified cover~\eqref{dc}.  By construction, the loop~$c$ lies in
the spherical cap~$\hat{\C}^+_\delta$ of~$\hat{\C}^+$,
\cf~Definition~\ref{51}.  Here, the hemisphere~$\hat{\C}^+$ is endowed
with the quotient metric from~$\DD_{\leq 0}$.  Furthermore, the
loop~$c$ surrounds more than one branch point in~$\hat{\C}^+$,
otherwise its lift~$\gamma$ would be contractible in~$\DD_{\leq 0}$.

Suppose that~$c$ is not simple.
Every arc of~$c$ forming a simple loop surrounds exactly one branch point of~\eqref{dc} in~$\hat{\C}^+$.
Indeed, if it surrounded exactly two branch points, it would lift to a noncontractible loop of~$\DD_{\leq 0}$ shorter than~$\gamma$, which is impossible.
If it surrounded three branch points, it would be double covered by a loop of~$\DD_{\leq 0} \setminus \mathcal{M}$ homotopic to~$\partial \mathcal{M}$ and so of length at least~$1$, which is impossible since the systole of~$\DD_{\leq 0}$ is at most~$1$.
Now, since~$c$ surrounds more than one branch point, there exist two arcs of~$c$ forming two simple loops surrounding two different branch points.
From these two simple loops, we can form with the shorter path of~$c$ joining them a loop in~$\hat{\C}^+$ homotopic to a simple loop surrounding exactly two branch points.
By smoothing out its corners, the loop we just formed can be made shorter than~$c$.
This yields a contradiction as it lifts to a noncontractible loop of~$\DD_{\leq 0}$.

In conclusion, the loop~$c$ is simple and surrounds at least two
branch points.  Now, if it surrounds three branch points, we already
showed that it is of length at least~$1$.  If it surrounds exactly two
branch points, its lift~$\gamma$ is homotopic to one of the three
geodesic loops of length~$1$ made of the two segments connecting a
pair of Weierstrass points.  In this case, the length of~$\gamma$ is
also equal to~$1$ since the metric is nonpositively curved.
\end{proof}

\section{Other decompositions are not optimal}
\label{six}

In this section, we complete the proof of Theorem~\ref{theo:main1} by
showing that the other configurations for~$\Gamma$,
\cf~Definition~\ref{51}, correspond to nonpositively curved Dyck's
surfaces with larger area.  \\

We start with some area estimates on the Voronoi cells of a
nonpositively curved Dyck's surface~$3\RP^2$ with unit systole, whose
metric is invariant by its hyperelliptic involution.

\begin{lemma} 
\label{lem:pi}
Every Voronoi cell of~$3\RP^2$ has area at least
\[
\pi h^2 \simeq 0.15887.
\]
\end{lemma}

\begin{proof}
From Lemma~\ref{lem:dist}, the centers of
the Voronoi cells are at distance at least~$\frac{1}{2}$ from each
other.  In particular, every Voronoi cell contains an embedded disk of
radius~$h<\frac{1}{4}$.  Since the metric is nonpositively curved the
area of this disk is at least~$\pi h^2$.
\end{proof}

\begin{lemma} \label{lem:x}
A Voronoi cell of~$3\RP^2$ whose projection to~$\hat{\C}^+$ has an edge of length~$x$ lying in~$c_\delta$ has area at least~$2hx$.
\end{lemma}

\begin{proof}
The comparison Euclidean polygon of such a Voronoi cell is a convex polygon of~$\R^2$, symmetric with respect to its center~$O$, with two opposite sides of length~$x$ at distance at least~$h$ from~$O$.
These two opposite sides span a parallelogram lying in the Euclidean polygon, which clearly satisfies the desired area lower bound.
\end{proof}

\begin{lemma} \label{lem:2edges}
A Voronoi cell of~$3\RP^2$ whose projection to~$\hat{\C}^+$ is bounded by exactly two edges has area at least~$h$.
\end{lemma}

\begin{proof}
The comparison Euclidean polygon of such a Voronoi cell has four sides.
It is a parallelogram with two opposite sides at distance at least~$2h$ from each other; the other two sides are at distance at least~$1$ from each other.
Hence its area is at least~$h$.
\end{proof}

Recall that the graph~$\Gamma$ (see Section~\ref{sec:extremal},
Definition~\ref{51}) has at most four vertices.  Clearly, the valence
of each vertex is at least~$3$ and at least one of the vertices lies
in~$c_\delta$.  We will consider four cases based on the number of
vertices lying in~$c_\delta$.

Before starting our discussion, we observe that a face of~$\hat{\C}^+_\delta$ cannot be bounded by a single edge of~$\Gamma$, otherwise its comparison Euclidean polygon would be bounded by two halfspaces, which is impossible.
This observation and the restriction on the number of faces of~$\hat{\C}^+_\delta$ will be implicitely used in the description of the different cases below. \\

\underline{Case 1:} Suppose that only one vertex lies in~$c_\delta$.
By assumption, there is a Voronoi cell of~$3\RP^2$ whose projection to~$\hat{\C}^+$ has~$c_\delta$ as an edge.
From Lemma~\ref{lem:x}, the area of this cell is at least
\[
2h \, \length(c_\delta) \geq 2h.
\]
Since the area of each of the other two cells of~$3\RP^2$ is at least~$\pi h^2$, \cf~Lemma~\ref{lem:pi}, we obtain using Proposition~\ref{prop:U} that
\[
\area(3\RP^2) \geq 2 \delta + 2h + 2\pi h^2 = 1+ 2\pi h^2 > \area \DD_{\leq0}.
\]

\underline{Case 2:} Suppose that exactly two vertices lie in~$c_\delta$.
The two edges of~$c_\delta$, of length~$x$ and~$y$, are part of two different faces of~$\hat{\C}^+_\delta$.
From Lemma~\ref{lem:x}, the total area of the two corresponding Voronoi cells of~$3\RP^2$ is at least~$2hx+2hy=2h$.
Since the area of the third cell of~$3\RP^2$ is at least~$\pi h^2$, \cf~Lemma~\ref{lem:pi}, we conclude using Proposition~\ref{prop:U} that
\[
\area(3\RP^2) \geq 2 \delta + 2h + \pi h^2 = 1+ \pi h^2 > \area \DD_{\leq0}.
\]

\underline{Case 3:} Suppose that exactly three vertices lie in~$c_\delta$.

If a fourth vertex lies in the interior of~$\hat{\C}^+_\delta$, then we are in the situation already described in Section~\ref{sec:extremal}. 

If there is no fourth vertex, two faces of~$\hat{\C}^+_\delta$ are bounded by exactly two edges.
By Lemma~\ref{lem:2edges}, the total area of the two corresponding Voronoi cells in~$3\RP^2$ is at least~$2h$.
We conclude as in Case 2. \\

\underline{Case 4:} Suppose that four vertices lie in~$c_\delta$, which is the maximal number of vertices of~$\Gamma$.
In this case, two faces of~$\hat{\C}^+_\delta$ are bounded by exactly two edges and we conclude as in Case 3. \\

This proves that the optimal systolic inequality for nonpositively
curved Dyck's surfaces is given by~\eqref{eq:extremal}, where the
equality case is attained by the surface~$\DD_{\leq 0}$ described at
the end of Section~\ref{sec:extremal}.

\forget
\section{Automorphism group of the extremal Dyck's surface}

The following proposition provides a description of the automorphism and symmetry groups of~$\DD_{\leq0}$ viewed as a Riemannian Klein surface.
We also present the corresponding groups for its orientable double cover~$\hat{\DD}_{\leq0}$.


\begin{proposition} \label{prop:sym}
\mbox{ }
\begin{enumerate}
\item
The automorphism group and the symmetry group of~$\DD_{\leq0}$ are
both isomorphic to
\[
D_3 \times \Z/2\Z.
\]
\item 
The holomorphic and antiholomorphic automorphism group and the 
symmetry group of~$\hat{\DD}_{\leq0}$ are both isomorphic to
\[
D_3 \times (\Z/2\Z)^2.
\]
\end{enumerate}
In particular, the orientable double cover of~$\DD_{\leq0}$ is not
conformally equivalent to Bolza's curve.
\end{proposition}

\begin{proof}
The natural homomorphism between the automorphism (resp. symmetry)
group of~$\DD_{\leq0}$
and the permutation group~$D_3$ of its Weierstrass points is
surjective.  Its kernel is composed of dianalytic (\ie locally
holomorphic or antiholomorphic) automorphisms preserving the
Weierstrass points.  The only two automorphisms with this property are
the identity map and the hyperelliptic involution (which is an
isometry).  Since they commute with every holomorphic map (and so
every isometry), \cf~\cite[\S III.9]{FK},
we deduce that the automorphism (resp. symmetry) group
of~$\DD_{\leq0}$ is isomorphic to
\[
D_3 \times \Z/2\Z
\]
(see Subsection~\ref{24} for a discussion of the related automorphism
group of the inscribed prism).

Similarly, the symmetry group of~$\hat{\DD}_{\leq0}$ acts by isometries on the right triangular prism inscribed in~$S^2$ whose vertices are the branch points of the holomorphic ramified double cover over~$S^2$, \cf~Subsection~\ref{24}.
Since the isometry group of this prism is isomorphic to~$D^3 \times \Z/2\Z$, we obtain an epimorphism 
\[
\Isom(\hat{\DD}_{\leq0}) \to D_3 \times \Z/2\Z
\]
whose kernel is generated by the hyperelliptic involution, \cf~\cite[\S III.9]{FK}.
As the hyperelliptic involution commutes with every holomorphic and antiholomorphic map, we deduce that the symmetry group of~$\hat{\DD}_{\leq0}$
is isomorphic to
\[
D_3 \times (\Z/2\Z)^2.
\]
Now, the metric on~$\hat{\DD}_{\leq0}$ averaged by the holomorphic and antiholomorphic
automorphism group of~$\hat{\DD}_{\leq0}$ descends to a metric~$g$
on~$3 \RP^2$.  The proofs of Lemmas~\ref{lem:1} and~\ref{lem:2} show
that this new metric~$g$ is nonpositively curved with a systolic ratio
at least as good as the one of~$\DD_{\leq0}$.
Since~$\hat{\DD}_{\leq0}$ is extremal, we deduce from the equality
case between the systolic ratios that~$g$ agrees with the extremal
metric on~$\DD_{\leq0}$.  Thus, the holomorphic and antiholomorphic automorphism group
of~$\hat{\DD}_{\leq0}$ is contained in its 
symmetry group.  As the opposite inclusion is clear, we derive the
desired isomorphism.
\end{proof}
\forgotten

\section{Conformal classes of extremal Dyck's surfaces}
\label{nine}

In this section, we compare the conformal classes of three genus two
Riemann surfaces that are significant for various systolic extremality
problems.  More specifically, we consider the Bolza
surface~$\mathcal{B}$ and the two orientable double covers of the
extremal hyperbolic Dyck's surface~$\DD_{-1}$ and the extremal
nonpositively curved Dyck's surface~$\DD_{\leq 0}$.

\subsection{The Bolza and Dyck's surfaces}

The following well-known result shows that the Bolza surface is not
the orientable double cover of any Dyck's surface.

\begin{proposition} 
\label{prop:bolza}
The Bolza surface~$\mathcal{B}$ admits no fixed point-free
antiholomorphic involution.
\end{proposition}

\begin{proof}
Suppose there is a fixed point-free antiholomorphic involution~$\tau$
on~$\mathcal{B}$.  Recall that $\tau$ commutes with the hyperelliptic
involution~$J$.  Both quotients $S^2=\mathcal{B}/J$ and
$3\RP^2=\mathcal{B}/\tau$ admit a ramified double cover over the
hemisphere~$\hat{\C}^+$, \cf~Section~\ref{sec:conf}.  The boundary
of~$\hat{\C}^+$ is double covered by a loop of~$\mathcal{B}$ on which
both~$J$ and~$\tau$ act as an antipodal map.  

In the setting of Section~\ref{sec:conf}, the projection of this loop to~$S^2$ is an equator with no branch
point lying in it.  The quotient map $\bar{\tau}:S^2 \to S^2$ induced
by~$\tau$ is an anticonformal map which fixes pointwise the equator
and switches the two hemispheres.  Thus, the composition~$\bar{\tau} \circ \varphi$
of~$\bar{\tau}$ with the reflexion~$\varphi$ along the equator conformally acts
on each hemisphere fixing pointwise the equator.  

By applying the Cauchy integral formula to the holomorphic
map~\mbox{$\bar{\tau} \circ \varphi$} on each hemisphere or simply by complex analytic continuation, we deduce that the composition~$\bar{\tau} \circ \varphi$ is
the identity map.  Thus, the map~$\bar{\tau}$ is the reflexion along
this equator containing no branch point.  Since $\bar{\tau}$ preserves
the branch points of the ramified double cover $\mathcal{B} \to S^2$,
we derive a contradiction.  Indeed, the branch points of the Bolza
surface form a regular octahedron on~$S^2$ and all the reflexions
of~$S^2$ acting on this octahedron fix at least one vertex, which is
not the case of~$\bar{\tau}$.
\end{proof}

\subsection{Conformal collar and capacity}
We introduce a conformal invariant which will allow us to distinguish the conformal classes of the extremal
nonpositively curved Dyck's surface~$\DD_{\leq 0}$ and the extremal
hyperbolic Dyck's surface~$\DD_{-1}$. \\

Recall that the extremal hyperbolic Dyck's surface~$\DD_{-1}$ described
in~\cite{Par} and~\cite{Gen} is obtained by identifying opposite pairs
of points on the boundary component of the maximal hyperbolic surface
of signature~$(1,1)$ with boundary length
\[
\ell = 2 \, {\rm arccosh} \left( \frac{5+\sqrt{17}}{2} \right) \simeq
4.397146,
\]
\cf~\cite[p.~578]{sch}.  The term ``maximal" refers to a hyperbolic
surface with fixed geodesic boundary length whose systole is maximal.
The systole of the extremal hyperbolic Dyck's surface~$\DD_{-1}$ is
equal to~$\ell/2$.  Furthermore,~$\DD_{-1}$ has the same isometry
group~$G$ as~$\DD_{\leq 0}$, which is isomorphic to~$D_3 \times
\Z/2\Z$, \cf~Proposition~\ref{prop:sym}.

\begin{definition}
Consider a~$G$-invariant conformal structure on Dyck's surface~$3\RP^2$ ({\it e.g.}, the conformal structure of~$\DD_{\leq 0}$ or~$\DD_{-1}$).
\forget 
The order two element in the center of~$G$ is the
hyperelliptic involution.  The axes of the order two symmetries in the
subgroup~$D_3$ meet at exactly two points~$p$ and~$q$ in~$3\RP^2$,
\cf~Figure~\ref{fig:b}.  Cut open~$3\RP^2$ along the three segments of
these axes with endpoints~$p$ and~$q$ which pass through a Weierstrass
point.%
\footnote{It seems to me there should be a simpler description of this
on the sphere: cut along the ``Y''-shape obtained by connecting the
north pole to each of the three Weierstrass points.  In other words,
connect the vertex of the dual cube to the three adjacent vertices of
the octahedron. Then the picture is lifted to the surface.}
\forgotten 
The union of the fixed-point sets of the order two automorphisms of~$3\RP^2$ defines a graph.
By removing the edges of this graph which meet the ramification locus of~\eqref{dc} at non-Weierstrass points, we obtain a graph~$\Gamma$ on~$3 \RP^2$.
The \emph{collar (or annulus) corresponding to the conformal structure of~$3\RP^2$} is the orientable double cover~$\mathcal{A}$ of the surface obtained by cutting~$3\RP^2$ open along~$\Gamma$.
By definition, it only depends on the conformal structure of the Dyck's surface.

For example, the graph~$\Gamma$ on the extremal nonnegatively curved Dyck's surface agrees with the outer boundary component of~$\mathcal{H}$ after identification of the opposite sides, \cf~Figure~\ref{fig:0}.

The \emph{soul} of the open collar~$\mathcal{A}$ is the simple loop~$C$ defined as the ramification locus of
\[
\mathcal{A} \subset \Sigma^2 \to 3\RP^2 \to {\hat \C}^+
\]
(since the collar is open the Weierstrass points are excluded), see~\eqref{dc}.  
\end{definition}


Let us recall the following definition.

\begin{definition}
The \emph{capacity} of a Riemannian collar~$\mathcal{A}$, with
boundary components~$\partial \mathcal{A}_-$ and~$\partial
\mathcal{A}_+$, is defined as
\begin{equation} \label{eq:cap}
\capa \mathcal{A} = \inf_u \int_\mathcal{A} |\nabla u|^2
\end{equation}
where~$u$ runs over piecewise smooth functions on~$\mathcal{A}$
with~$u=0$ on~$\partial \mathcal{A}_-$ and~$u=1$ on~$\partial
\mathcal{A}_+$.

The infimum is attained by the unique harmonic function satisfying the
boundary conditions.  The capacity is a conformal invariant.
\end{definition}

\begin{remark}
\label{rem:cap}
By construction, the capacity of the collar corresponding to a $G$-invariant conformal structure of~$3\RP^2$ is a conformal invariant of the surface.
\end{remark}

In the rest of the article, we estimate the collar capacities for the extremal nonpositively curved Dyck's surface~$\DD_{\leq 0}$ and the extremal hyperbolic Dyck's surface~$\DD_{-1}$

\subsection{Collar capacity for~$\DD_{\leq 0}$}
In the following proposition, we bound from above the capacity of the collar corresponding to the extremal nonpositively curved Dyck's surface.

\begin{proposition} 
\label{prop:cap1}
Let~$\mathcal{A}_{\leq 0}$ be the collar corresponding to the conformal structure of the extremal nonpositively curved Dyck's surface~$\DD_{\leq 0}$.
Then
\[
\capa \mathcal{A}_{\leq 0} \leq 2.28308.
\]
\end{proposition}

\begin{proof}
The surface~$\DD_{\leq 0}$ is tiled by three flat ``hexagons'' (flat
hexagons with two flat rectangles attached to them) centered at the Weierstrass points, \cf~Figure~\ref{fig:b}.
By construction, the collar~$\mathcal{A}_{\leq 0}$ (composed of a cylinder and half hexagons) has piecewise geodesic boundary components.

The simple loop~$C$ decomposes the collar~$\mathcal{A}_{\leq 0}$ into two regions~$\mathcal{A}_{\leq 0}^+$ and~$\mathcal{A}_{\leq 0}^-$.
We define a piecewise smooth function~$u$ on~$\mathcal{A}_{\leq 0}$ as follows
\[
u(x) = 
\left\{
\begin{array}{ll}
\min\{\frac{1}{2}+d(x,C),1\} & \mbox{ if } x \in \mathcal{A}_{\leq 0}^+ \\
 & \\
\max\{\frac{1}{2}-d(x,C),0\} & \mbox{ if } x \in \mathcal{A}_{\leq 0}^-
\end{array}
\right.
\]
Since the points on the boundary components of~$\mathcal{A}_{\leq 0}$ are at distance at least~$\frac{1}{2}$ from~$C$, the function~$u$ is a test function for the capacity of~$\mathcal{A}_{\leq 0}$, \cf~\eqref{eq:cap}.
Thus,
\begin{eqnarray}
\capa \mathcal{A}_{\leq 0} & \leq & \int_{\mathcal{A}_{\leq 0}} | \nabla u |^2 \nonumber \\
& \leq & \area\{ x \in \mathcal{A}_{\leq 0} \mid d(x,C) \leq \frac{1}{2} \} \label{eq:1/2}.
\end{eqnarray}

\noindent There is a unique minimizing ray~$r_x$ from every point~$x$ of~$\mathcal{A}_{\leq 0}$ to~$C$.
The points~$x$ of~$\mathcal{A}_{\leq 0}$ such that~$r_x$ passes through a given conical singularity~$x_0$ form a symmetric flat quadrilateral~$Q_{x_0}$, with two right angles from which two edges of length~$h$ meeting at~$x_0$ with an angle~$\theta$ arise, \cf~Figure~\ref{fig:d}.
Recall that the constants $h$ and~$\theta$ are defined in~\eqref{eq:h}.

\begin{figure}[ht]
\setlength{\unitlength}{1pt}

\begin{picture}(0,0)(0,0)
\put(-45,-70){$h$}
\put(30,-70){$h$}
\put(0,-85){$\frac{\theta}{2}$}
\put(-8,-122){$x_0$}
\end{picture}

\includegraphics[height=4cm]{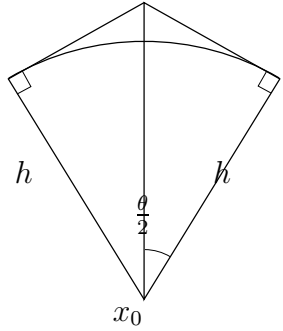}

\caption{The quadrilateral~$Q_{x_0}$} \label{fig:d}
\end{figure}

\noindent The points of~$Q_{x_0}$ outside the disk of radius~$h$ centered at~$x_0$ are at distance at least~$\frac{1}{2}$ from~$C$.
They form a region of area 
\[
\left[ \tan \left( \frac{\theta}{2} \right) - \frac{\theta}{2} \right] \, h^2.
\]
Furthermore, the quadrilaterals~$Q_{x_0}$ are disjoint as~$x_0$ runs over the conical singularities of~$\mathcal{A}_{\leq 0}$.
Continuing with~\eqref{eq:1/2}, we obtain the following upper bound for the capacity of~$\mathcal{A}_{\leq 0}$ using~\eqref{eq:A}
\begin{eqnarray*}
\capa \mathcal{A}_{\leq 0} & \leq & 2 \area \DD_{\leq0} - 12 \, \left[ \tan \left (\frac{\theta}{2} \right) - \frac{\theta}{2} \right] \, h^2 \\
 & \leq & 2.28308.
\end{eqnarray*}
\end{proof}

\subsection{A general lower bound on the capacity of a collar}
We will need the following lower bound on the capacity of a collar.
This bound was established by B. Muetzel~\cite[Lemma~2.2]{mue} in a
more general form.  We include a proof for the reader's convenience.

\begin{lemma}[B.~Muetzel]
\label{lem:capa}
Consider a hyperbolic collar~$\mathcal{A}$ around a closed geodesic
loop of length~$\ell$ parametrized in Fermi coordinates by
\[
\{ (t,s) \mid t \in [0,\ell), s \in (a(t),b(t)) \}.
\]
Then
\[
\capa \mathcal{A} \geq \, \int_0^\ell \frac{1}{H(b(t)) - H(a(t))} \, dt
\]
where~$H(s) = 2 \, {\rm arctan}(\exp(s))$.
\end{lemma}

\begin{proof}
In Fermi coordinates, the hyperbolic metric on~$\mathcal{A}$ can be expressed as~$g = \cosh(s)^2 \, dt^2 + ds^2$, \cf~\cite{Bus}.
Let~$\xi$ be the unit vector field on~$\mathcal{A}$ induced by~$\frac{\partial}{\partial s}$.
For every piecewise smooth function~$u$ on~$\mathcal{A}$ with~$u=0$ on~$\partial \mathcal{A}_-$ and~$u=1$ on~$\partial \mathcal{A}_+$, we have
\[
\int_{\mathcal{A}} | \nabla u |^2 \geq \int_{\mathcal{A}} g(\nabla u , \xi)^2 = \int_0^\ell \int_{a(t)}^{b(t)} \left( \frac{\partial u}{\partial s} \right)^2  \cosh(s) \, ds \, dt.
\]

Given a continuous function~$h:[a,b] \to (0,\infty)$ (in our case,~$h(s) = \cosh(s)$), we want to minimize the integral
\[
\int_a^b f'(s)^2 \, h(s) \, ds
\]
where~$f:[a,b] \to \R$ is a piecewise smooth function with~$f(a)=0$ and~$f(b)=1$.
Let~$H$ be a primitive of~$\frac{1}{h}$.
Making the change of variables~$\tau = H(s)$, we obtain
\begin{align}
\int_a^b f'(s)^2 \, h(s) \, ds & = \int_{H(a)}^{H(b)} \left[ f'(H^{-1}(\tau)) \cdot h(H^{-1}(\tau)) \right]^2 \, d\tau \nonumber \\
& = \int_{H(a)}^{H(b)} (f \circ H^{-1})'(\tau)^2 \, d\tau \label{eq:fh}
\end{align}
since~$(H^{-1})'(\tau) = h(H^{-1}(\tau))$.
%
By the Cauchy-Schwarz inequality, we have
\[
1=\left( \int_{H(a)}^{H(b)} (f \circ H^{-1})'(\tau) \, d\tau \right)^2 \leq (H(b)-H(a)) \, \int_{H(a)}^{H(b)} (f \circ H^{-1})'(\tau)^2 \, d\tau.
\]
Hence,
\[
\int_a^b f'(s)^2 \, h(s) \, ds \geq \frac{1}{H(b) - H(a)}.
\]
Therefore,
\[
\int_{\mathcal{A}} | \nabla u|^2 \geq \int_0^\ell \frac{dt}{H(b(t)) - H(a(t))}.
\]
\end{proof}

\subsection{Collar capacity for~$\DD_{-1}$}
In the following proposition, we bound from below the capacity of the collar corresponding to the extremal hyperbolic Dyck's surface.

\begin{proposition} 
\label{prop:cap2}
Let~$\mathcal{A}_{-1}$ be the collar
corresponding to the conformal structure of the extremal hyperbolic Dyck's surface~$\DD_{-1}$.
Then
\[
\mathcal{A}_{-1} \geq 2.29461.
\]
\end{proposition}

\begin{proof}
The surface~$\DD_{-1}$ is tiled by three hyperbolic hexagons centered at the Weierstrass points, \cf~\cite{Gen}.  
By construction, the collar~$\mathcal{A}_{-1}$ is made of half
hexagons with one side lying in~$C$ and has piecewise geodesic
boundary components.

Each tiling hyperbolic hexagon of~$\DD_{-1}$ decomposes into four
isometric trirectangles with an acute angle equal to~$\frac{\pi}{3}$.
The sides~$a$ and~$b$ opposite to the acute angle are of
length~$\ell/4$ and~$\ell/12$, with the shorter side~$b$ lying in~$C$.
Observe that~$\mathcal{A}_{-1}$ is composed of exactly~$24$ such
trirectangles~$T$.  From the hyperbolic formula~\cite[p.~454,
2.3.1(iv)]{Bus} for trirectangles, the geodesic arc of~$T$ orthogonal
to~$b$ at the point at distance~$t$ from~$a$, \cf~Figure~\ref{fig:c},
is of length
\begin{equation} 
\label{eq:at}
\text{arctanh}\left[\cosh(t)\,\tanh\left(\frac{\ell}{4}\right)\right].
\end{equation}

\begin{figure}[ht]
\setlength{\unitlength}{1pt}

\begin{picture}(0,0)(0,0)
\put(-47,-60){$a$}
\put(15,-112){$b$}
\put(-25,-122){$t$}
\put(22,-19){$\frac{\pi}{3}$}
\end{picture}

\includegraphics[height=4cm]{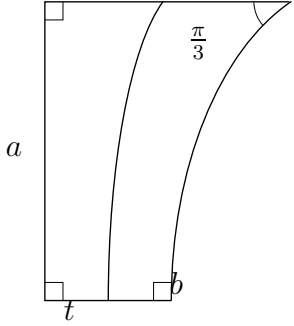}

\caption{The trirectangle~$T$} \label{fig:c}
\end{figure}

\noindent In Fermi coordinates, the collars~$\mathcal{A}_{-1}$ is parametrized by 
\[
\{ (t,s) \mid t \in [0,\ell), s \in (-a(t),a(t)) \}
\]
where~$a(t)$ agrees with~\eqref{eq:at} for~$t \in [0,\frac{\ell}{12})$
(the other values can be derived by symmetry).  From
Muetzel's~Lemma~\ref{lem:capa}, we have
\begin{eqnarray*}
\capa \mathcal{A}_{-1} & \geq & 12 \, \int_0^\frac{\ell}{12} \frac{1}{H(a(t)) - H(-a(t))} \, dt \\
& \geq &  2.29461
\end{eqnarray*}
where~$H(s) = 2 \, {\rm arctan}(\exp(s))$.
\end{proof}

From Proposition~\ref{prop:cap1}, Proposition~\ref{prop:cap2} and Remark~\ref{rem:cap}, we immediately derive the following result.

\begin{corollary}
The extremal nonpositively curved Dyck's
surface~$\DD_{\leq 0}$ is not conformally equivalent to the extremal
hyperbolic Dyck's surface~$\DD_{-1}$.
\end{corollary}

\begin{remark}
Simpler bounds on the capacities of the collars can be derived both
for~$\mathcal{A}_{-1}$ and~$\mathcal{A}_{\leq 0}$ as follows.  The
collar~$\mathcal{A}_{-1}$ can be isometrically embedded into the
bi-infinite hyperbolic cylinder with~$C$ as a simple geodesic loop.
In this cylinder, the collar~$\mathcal{A}_{-1}$ is contained in the
tubular neighborhood~$U$ of~$C$ of width the length of the side
opposite to~$a$ in the trirectangle~$T$, \cf~Figure~\ref{fig:c}.  We
deduce that the capacity of the collar~$\mathcal{A}_{-1}$ is bounded
from below by the capacity of~$U$ for which a formula has been
established by Buser and Sarnak~\cite[p.~37]{BS94}.  Even more
directly, the capacity of the collar~$\mathcal{A}_{\leq 0}$ is bounded
from above by twice the area of~$\DD_{\leq 0}$.  However, none of
these estimates is strong enough for our purpose.  This explains why
we made use of finer estimates.
\end{remark}

\section*{Acknowledgments}

We are grateful to Bjoern Muetzel, Hugo Parlier and Robert Silhol for helpful comments.



\begin{thebibliography}{Ai}

\forget
\bibitem{AK} Ambrosio, L.; Katz, M.: Flat currents modulo~$p$ in metric spaces and filling radius inequalities, 
{\em Comm. Math. Helv.\/} \textbf{86} (2011), no.~3, 557--591.  
See http://dx.doi.org/10.4171/CMH/234 and arXiv:1004.1374.
  
\bibitem{BB10} Babenko, I.; Balacheff, F.: Distribution of the systolic volume of homology classes.  See arXiv:1009.2835.

\bibitem{BPS} Balacheff, F.; Parlier, H.; Sabourau, S.: Short loop decompositions of surfaces and the geometry of Jacobians. 
{\em Geometric and Functional Analysis (GAFA)\/}, to appear.
See arXiv:1011.2962.

\bibitem{BCIK1} Bangert, V; Croke, C.; Ivanov, S.; Katz, M.:
Filling area conjecture and ovalless real hyperelliptic surfaces.
{\em Geometric and Functional Analysis (GAFA)\/} \textbf{15} (2005)
no.~3, 577-597.  See arXiv:math.DG/0405583
\forgotten



\bibitem{BK1} Bangert, V.; Katz, M.; Stable systolic inequalities and
cohomology products.  {\em Comm. Pure Appl. Math.\/} \textbf{56}
(2003), 979--997.  See arXiv:math.DG/0204181.

\bibitem{BK2} \bysame, An optimal Loewner-type systolic inequality and
harmonic one-forms of constant norm.  {\em Comm. Anal. Geom.\/}
\textbf{12} (2004), no.~3, 703-732.  arXiv:math.DG/0304494


\bibitem{Bav1} Bavard, C.: In\'egalit\'e isosystolique pour la
bouteille de Klein.  {\em Math. Ann.} \textbf{274} (1986), no.~3,
439--441.

\bibitem{bav92} Bavard, C.: In\'egalit\'es isosystoliques conformes, {\em Comment. Math. Helv.} {\bf 67} (1992), no.~1, 146--166.


\bibitem{bol} Bolza, O.: On binary sextics with linear transformations into themselves. {\em Amer. J. Math.} {\bf 10} (1887) 47--70.

\bibitem{Bl} Blatter, C.: \"{U}ber Extremall\"{a}ngen auf
geschlossenen Fl\"{a}chen.  \textit{Comment.\ Math.\ Helv.}
\textbf{35} (1961), 153--168.


\bibitem{Bl2} Blatter, C.: Zur Riemannschen Geometrie im
Grossen auf dem M\"{o}bius\-band.  {\em Compositio Math.} \textbf{15}
(1961), 88--107.


\bibitem{BI94} Burago, D.; Ivanov, S.: Riemannian tori without
conjugate points are flat.  {\em Geom. Funct. Anal.\/} \textbf{4}
(1994), no.~3, 259--269.

\bibitem{BI95} \bysame, On asymptotic volume of tori.  {\em
Geom. Funct. Anal.\/} \textbf{5} (1995), no.~5, 800--808.

\bibitem{Bus} Buser, P.: Geometry and spectra of compact Riemann
surfaces.  Reprint of the 1992 edition. Modern Birkh\"auser
Classics. Birkh\"auser Boston, Inc., Boston, MA, 2010.

\bibitem{BS94}
Buser, P.; Sarnak, P.:
On the period matrix of a Riemann surface of large genus. With an appendix by J. H. Conway and N. J. A. Sloane.  
Invent. Math.  117  (1994),  no. 1, 27--56.

\bibitem{cal}
Calabi, E.:  Extremal isosystolic metrics for compact surfaces.
Actes de la Table Ronde de G\'eom\'etrie Diff\'erentielle, S\'emin. Congr. 1 (1996), Soc. Math. France, 146--166.

\forget
\bibitem{DKR2} Dranishnikov, A.; Katz, M.; Rudyak, Y.: Cohomological dimension, self-linking, and systolic geometry.  
{\em Israel Journal of Math.\/} \textbf{184} (2011), no.~1, 437--453.  See arXiv:0807.5040

\bibitem{Elm10} El Mir, C.: Conformal isosystolic inequality of Bieberbach 3-manifolds.  See arXiv:1007.0877
\forgotten

\bibitem{FK} Farkas, H. M.; Kra, I.: Riemann surfaces.  Second
edition.  {\em Graduate Texts in Mathematics\/}
\textbf{71}. Springer-Verlag, New York, 1992.


\bibitem{Gen} Gendulphe, M.: Paysage Systolique Des Surfaces
Hyperboliques Compactes De Caracteristique -1.  See arXiv:math/0508036

\bibitem{Gr4} Gromov, M.: Metric structures for Riemannian and
non-Riemannian spaces.  Based on the 1981 French original. With
appendices by M. Katz, P. Pansu and S. Semmes.  Translated from the
French by Sean Michael Bates.  Reprint of the 2001 English edition.
Modern Birkh\"auser Classics.  Birkh\"auser Boston, Inc., Boston, MA,
2007.


\bibitem{HKK} Horowitz, C.; Katz, Karin Usadi; Katz, M.: Loewner's
torus inequality with isosystolic defect.  {\em Journal of Geometric
Analysis} {\bf 19} (2009), no.~4, 796-808.  See arXiv:0803.0690


\bibitem{Jen} Jenni, F.: \"Uber den ersten Eigenwert des
Laplace-Operators auf ausgew\"ahlten Beispielen kompakter Riemannscher
Fl\"achen.  {\em Comment. Math. Helv.} \textbf{59} (1984), no.~2,
193--203.

\forget
\bibitem{KK2} Katz, K.; Katz, M.: Bi-Lipschitz approximation by
finite-dimensional imbeddings.  {\em Geometriae Dedicata\/}
\textbf{150} (2010) no.~1, 131--136.  Available at the site
arXiv:0902.3126

\bibitem{KW} Katz, K.; Katz, M.; Sabourau, S.; Shnider, S.;
Weinberger, Sh.: Relative systoles of relative-essential
$2$-complexes.  {\em Algebraic and Geometric Topology\/} \textbf{11}
(2011), 197-217.  See arXiv:0911.4265.
\forgotten


\bibitem{SGT} Katz, M.: Systolic geometry and topology.  With an
appendix by Jake P. Solomon.  {\em Mathematical Surveys and
Monographs}, \textbf{137}.  American Mathematical Society, Providence,
RI, 2007.


\bibitem{KS} Katz, M.; Sabourau, S.: An optimal systolic inequality
for~$CAT(0)$ metrics in genus two.  \emph{Pacific J. Math.}
\textbf{227} (2007), no.~1, 95-107.



\bibitem{KS11} Katz, M.; Sabourau, S.: Hyperellipticity and systoles
of Klein surfaces.  \emph{Geometriae Dedicata} \textbf{159} (2012),
no.~1, 277--293.

\forget
\bibitem{KSV2} Katz, M.; Schaps, M.; Vishne, U.: Hurwitz quaternion
order and arithmetic Riemann surfaces.  \emph{Geom. Dedicata}
\textbf{155} (2011), no. 1, 151-161, see

http://dx.doi.org/10.1007/s10711-011-9582-3

and arXiv:math.RA/0701137


\bibitem{KSh} Katz, M.; Shnider, S.: Cayley 4-form comass and triality
isomorphisms.  {\em Israel J. Math.\/} \textbf{178} (2010), 187-208.
See arXiv:0801.0283
\forgotten


\bibitem{kkk} Klein, C.; Kokotov, A.; Korotkin, D. : Extremal
properties of the determinant of the Laplacian in the Bergman metric
on the moduli space of genus two Riemann surfaces. \emph{Math. Z.}
\textbf{261} (2009), no.~1, 73--108. See arXiv:math/0511217.






\bibitem{Mi} Miranda, R.: Algebraic curves and Riemann surfaces. {\em
Graduate Studies in Mathematics}, \textbf{5}.  American Mathematical
Society, Providence, RI, 1995.

\bibitem{mue} Muetzel, B.: Inequalities for the capacity of
non-contractible annuli on cylinders of constant and variable negative
curvature.  \emph{Geometriae Dedicata} (online first).  See
http://dx.doi.org/10.1007/s10711-012-9788-z 

and arXiv:1105.5060v1


\bibitem{Par} Parlier, H.: Fixed-point free involutions on Riemann
surfaces.  {\em Israel J. Math.} \textbf{166} (2008), 297-311.  See
arXiv:math.DG/0504109



\bibitem{Pu} Pu, P.M.: Some inequalities in certain nonorientable
Riemannian manifolds, {\it Pacific J. Math.\/} \textbf{2} (1952),
55--71.

\forget
\bibitem{Ry} Ryu, H.: Stable systolic category of the product of
spheres.  {\em Algebraic \& Geometric Topology\/} \textbf{11} (2011),
983-999.  See arXiv:1007.2913


\bibitem{Sa11} Sabourau, S.: Isosystolic genus three surfaces critical
for slow metric variations.  {\em Geometry \& Topology\/} \textbf{15}
(2011), 1477-1508.
\forgotten


\bibitem{Sak} Sakai, T.: A proof of the isosystolic inequality
for the Klein bottle.  {\em Proc. Amer. Math. Soc.} \textbf{104}
(1988), no.~2, 589--590.

\bibitem{sch} Schmutz, P.: Riemann surfaces with shortest geodesic of
maximal length.  {\em Geom. Funct. Anal.} \textbf{3} (1993), no.~6,
564--631.

\forget
\bibitem{Si} Silhol, R.: On some one parameter families of genus
2 algebraic curves and half twists. {\em Comment. Math. Helv.}
\textbf{82} (2007), no.~2, 413--449.


\bibitem{Si10} Silhol, R.: Actions of fractional Dehn twists on moduli
spaces.  Geometry of Riemann surfaces, 376--395.  {\em London
Math. Soc. Lecture Note Ser.\/} \textbf{368}, Cambridge Univ. Press,
Cambridge, 2010.
\forgotten

\end{thebibliography}
\end{document}